\documentclass[11pt]{amsart}
\usepackage[utf8]{inputenc}
\usepackage{bm}
\usepackage{fontenc}
\usepackage{amsfonts}
\usepackage{amssymb}
\usepackage{amsmath}
\usepackage{amsthm}\usepackage{mathtools}
\usepackage{enumerate}
\usepackage{enumitem} 
\usepackage[pagebackref,colorlinks,linkcolor=blue,citecolor=blue,urlcolor=blue,hypertexnames=false]{hyperref}
\usepackage{mathrsfs}
\usepackage{tikz}
\usetikzlibrary{calc}
\usepackage{marginnote}
\usepackage{xcolor,enumitem}
\usepackage{soul}
\usepackage[all,cmtip]{xy} \usepackage{caption}

\newcommand{\N}{\mathbb{N}}

\newcommand{\Lip}{\mathrm{Lip}}

\newtheorem{theorem}{Theorem} 
\newtheorem*{theorem*}{Theorem}
\newtheorem{proposition}[theorem]{Proposition}

\newtheorem*{proposition*}{Proposition}
\newtheorem{lemma}[theorem]{Lemma}
\newtheorem*{lemma*}{Lemma}
\newtheorem{corollary}[theorem]{Corollary}
\newtheorem*{corollary*}{Corollary}

\newtheorem*{fact*}{Fact}
\theoremstyle{definition}

\newtheorem*{definition*}{Definition}

\newtheorem*{acknowledgements*}{Acknowledgements}
\newtheorem*{disclaimer*}{Disclaimer on originality of the work}
\newtheorem*{leanformalization*}{Lean formalization}

\newtheorem*{claim*}{Claim}

\newtheorem*{conjecture*}{Conjecture}

\theoremstyle{remark}

\newtheorem*{example*}{Example}

\newtheorem*{remark*}{Remark}

\newtheorem*{note*}{Note}
\newtheorem*{question*}{Question}

\newcommand{\norm}[1]{\left\lVert #1 \right\rVert}

 \usepackage{anysize}
\usepackage{float}

\usepackage{anysize}
\marginsize{3.7cm}{3.7cm}{3.2cm}{3.2cm}

\begin{document}

\date{\today} % Activate to display a given date or no date

\title[Expanders prevent Property (H)]{Expanders prevent Property (H)}

\author[Braga]{Bruno M. Braga}
\address[B. M. Braga]{IMPA, Estrada Dona Castorina 110, 22460-320, Rio de Janeiro, Brazil}
\email{demendoncabraga@gmail.com}
\urladdr{\url{https://sites.google.com/site/demendoncabraga}}
\thanks {B. M. Braga  was partially supported by FAPERJ, grant E-26/200.167/2023,  by CNPq, grant 303571/2022-5, and by Serrapilheira, grant R-2501-51476.}

\author[Gartland]{C. Gartland}
\address[C. Gartland]{Department of Mathematics and Statistics, University of
North Carolina at Charlotte, Charlotte, North Carolina, USA}
\email{cgartla1@charlotte.edu}
\urladdr{\url{https://chrisgartland.wordpress.com/}}
\thanks{C. Gartland was partially supported by NSF Award DMS-2555144.}

\author[Lancien]{G. Lancien}
\address[G. Lancien]{Universit\'e Marie et Louis Pasteur, CNRS, LmB (UMR 6623), F-25000 Besan\c con, France.}
\email{gilles.lancien@univ-fcomte.fr}
%\urladdr{\url{ }}
\thanks{G. Lancien was partially supported by the French ANR project No. ANR-24-CE40-0892-01.}

\author[Motakis]{P. Motakis}
\address[P. Motakis]{Department of Mathematics and Statistics, York University, 4700 Keele Street, Toronto, Ontario, M3J 1P3, Canada}
\email{pmotakis@yorku.ca}
\urladdr{\url{https://pmotakis.mathstats.yorku.ca/}}
\thanks{P. Motakis was partially supported by NSERC Grant RGPIN-2021-03639.}

\author[Perneck\'a]{E. Perneck\'a}
\address[E. Perneck\'a]{Faculty of Information Technology, Czech Technical University in Prague, Th\'akurova 9, 160 00, Prague 6, Czech Republic}
\email{eva.pernecka@fit.cvut.cz}
%\urladdr{\url{ }}
\thanks{}

\author[Schlumprecht]{Th. Schlumprecht}
\address[Th. Schlumprecht]{Texas A\&M University, College Station, TX 77843,
USA and Faculty of Electrical Engineering, Czech Technical University in Prague, Zikova 4, 166 27, Prague}
\email{t-schlumprecht@tamu.edu}
\urladdr{\url{https://people.tamu.edu/~t-schlumprecht/}}
\thanks{Th. Schlumprecht was partially supported by NSF Award DMS-2349322.}
\begin{abstract}
    We show that if a sequence of expander graphs equi-coarsely embeds into a Banach space, then this Banach space  fails Kasparov and Yu's Property (H).  Consequently,  no Banach space with Property (H) can be coarsely universal for all countable groups. We provide a Lean verification of our results.
\end{abstract}

\maketitle

\section{Introduction}

The Novikov conjecture, which asserts the homotopy invariance of the higher
signatures of closed oriented manifolds, is one of the central open problems
in topology, and much of the progress on it has been geometric in nature: by
a theorem of Yu, it holds for every countable group which coarsely embeds
into a Hilbert space (see \cite[Corollary 1.2]{Yu2000} and \cite[Theorem 6.1]{SkandalisTuYu2002Top}). After Gromov constructed random groups that admit no coarse embedding into
a Hilbert space (\cite[Section 4.8]{Gromov2003GAFA}), there has been a sustained effort to
replace the Hilbert space in Yu's theorem by larger class of Banach spaces. Kasparov
and Yu carried this out in \cite{KasparovYu2012GeoTop}, where they introduced
a property of Banach spaces called \emph{Property (H)} and showed that the
Novikov conjecture holds for any countable group which coarsely embeds into a
Banach space satisfying it (\cite[Theorem 1.2]{KasparovYu2012GeoTop}). We recall that a Banach space $X$ has \emph{Property (H)} if there are a uniformly continuous function $F\colon S_X\to S_{\ell_2}$ and increasing sequences of finite dimensional Banach spaces $(X_n)_n$ and $(H_n)_n$ of $X$ and of the Hilbert space $\ell_2$, respectively, such that $X=\overline{\bigcup_nX_n}$, each $F(S_{X_n})\subseteq S_{H_n}$, and each $F_{\restriction S_{X_n}}\colon  S_{X_n}\to S_{H_n}$ is a degree one map. If the degree one condition is replaced by the weaker condition of having nonzero degree, then $X$ is said to have {\it rational Property (H)}. At present, all known examples of Banach spaces with Property (H) come from Mazur's modified maps constructed by Odell and Schlumprecht in \cite{OdellSchlumprecht1994Acta} (e.g., this is the case for all separable Banach lattices with nontrivial cotype, see \cite[Theorem 1.1]{ChengWang2018JMAA}).

The theorem of Kasparov and Yu raises the possibility of settling the Novikov conjecture in a single stroke, reducing the problem to showing that every countable group admits a coarse embedding into a Property (H) space. The following theorem and corollaries show that this is impossible.

\begin{theorem}\label{Thm.main}
Let $(X_n)_n$ and $(Y_n)_n$ be sequences of finite dimensional normed spaces such that $Y_n\subseteq L_1$ for all $n\in\N$. Let $(G_n=(V_n,E_n))_n$ be a sequence of expander graphs, and assume that there is a sequence of equi-coarse embeddings $(V_n\to X_n)_n$. Let $(F_n\colon S_{X_n}\to S_{Y_n})_n$ be a sequence of equi-uniformly continuous maps. Then, for all sufficiently large $n\in\N$, the maps $F_n$ have degree zero.
\end{theorem}

\begin{corollary}\label{Cor.Prop.H}
If a sequence of expander graphs equi-coarsely embeds into a Banach space, then the Banach space does not have Property (H) (or even rational Property (H)). In particular, if a Banach space contains the finite dimensional $(\ell_\infty^n)_n$ with uniform distortion (such as $c_0$), then it cannot have Property (H).
\end{corollary}

\begin{proof}
This follows from Theorem~\ref{Thm.main} and the fact that finite dimensional Hilbert spaces are isometrically contained in $L_1$.
\end{proof}

We remark that another recent work \cite{ChengChengWang} contains a proof that $c_0$ does not have rational Property (H), see \S\ref{disclaimer} below. We also note that it is currently unknown if a Banach space which does not contain the finite dimensional $(\ell_\infty^n)_n$ with uniform distortion can equi-coarsely contain a sequence of expander graphs.

\begin{corollary}\label{Cor.Group}
There exists a finitely generated group that does not coarsely embed into any Banach space with Property (H).
\end{corollary}

\begin{proof}
This follows from Corollary~\ref{Cor.Prop.H} the fact that there exists a finitely generated group containing a sequence of expander graphs isometrically (\cite[Theorem 4]{Osajda2020Acta}).
\end{proof}

We briefly remind the reader of the terminologies in Theorem \ref{Thm.main}. Firstly, given $k\in\N$ and $h>0$, a finite   graph $G=(V,E)$, where $V$ is the set of \emph{vertices} and $E$ the set of \emph{edges} of $G$, is called a \emph{$(k,h)$-expander} if each vertex has at most $k$ neighbors and 
\[|\{v\in V\setminus A\mid \exists u\in A, \ (v,u)\in E\}|\geq h|A|\]
for all $A\subseteq V$ with $|A|\leq |V|/2$. A \emph{sequence of expander graphs} is understood to be a sequence of $(k,h)$-expanders, for some $k\in\N$ and $h>0$, such that $\lim_n|V_n|=\infty$. As expanders must be connected, we view the vertex set of an expander as a metric space by endowing it with the shortest path distance, which we denote by $d_n$.  For  more on expander graphs, see the monograph \cite{LubotzkyBook2010}. We then say that   $(\varphi_n\colon V_n\to X_n)_n$ is a sequence of  \emph{equi-coarse embeddings} if there are $L>0$ and a nondecreasing $\rho\colon [0,\infty)\to[0,\infty)$ with $\lim_{t\to \infty}\rho(t)=\infty$  such that 
\[\rho(d_n(v,u))\leq \|\varphi_n(v)-\varphi_n(u)\|\leq Ld_n(v,u)\ \text{ for all }n\in\N\ \text{ and } v,u\in V_n.\]

\section{Proof of Main Result}

We now start presenting the proof of Theorem \ref{Thm.main}. The arguments are surprisingly simple and use just two applications of the characteristic Poincaré inequality satisfied by expanders: If $G=(V,E)$ is a $(k,h)$-expander, then  for any $f\colon V\to L_1$ we have
     \begin{equation}\label{PIII}
\frac{1}{|V|}\sum_{v\in V}\left\|f(v)-\frac{1}{|V|}\sum_{u\in V}f(u)\right\|\leq \frac{2k}{h}\Lip(f) \tag{\textbf{P.I.}}
  \end{equation}
     (see \cite[Theorem 4.7]{Ostrovskii2013Book}).
     
 Given a Banach space $X$, $B_X$ denotes its closed unit ball. We start with a lemma which modifies the equi-coarse embeddings $(V_n\to X_n)_n$ in Theorem \ref{Thm.main} to maps which better suit our purposes. 

\begin{lemma}\label{Lemma}
    Let $(X_n)_n$ be Banach spaces and $(G_n=(V_n,E_n))_n$ be a sequence of expander graphs. If there is a sequence of  equi-coarse embeddings $(V_n\to X_n)_n$, then there is a sequence of functions $(\psi_n\colon V_n\to B_{X_n})_n$ such that 
    \begin{equation}\label{Eq.prop.psi}
    \sum_{v\in V_n}\psi_n(v)=0, \ \liminf_{n\to \infty}\frac{1}{|V_n|}\sum_{v\in V_n}\|\psi_n(v)\|\geq 1\ \text{ and } \lim_{n\to\infty}\Lip(\psi_n)=0.
    \end{equation}
\end{lemma}

\begin{proof} 
Fix $k\in\N$ and $h>0$ such that each $G_n$ is a $(k,h)$-expander. For each $n\in\N$, let $d_n$ be the shortest path metric on $V_n$ and let $(\varphi_n\colon V_n\to {X_n})_n$ be a sequence of equi-coarse embeddings; so there is $L>0$ and $\rho\colon [0,\infty)\to [0,\infty)$ with $\lim_{t\to \infty}\rho(t)=\infty$ such that 
\[\rho(d_n(v,u))\leq \|\varphi_n(v)-\varphi_n(u)\|\leq Ld_n(v,u) \ \text{ for all }\ n\in\N\text{ and }\ v,u\in V_n.\]
By translating the $\varphi_n$'s  if needed,  assume that $\sum_{v\in V_n}\varphi_n(v)=0$.
We start by noticing that 
\begin{equation}\label{Eq.Antoinfty}A_n=\frac{1}{|V_n|}\sum_{v\in V_n}\|\varphi_n(v)\|\underset{n\to \infty}{\longrightarrow} \infty.\end{equation}
Indeed, as each $G_n$ is a $(k,h)$-expander, the condition on $k$ implies that for each $r>0$, the balls of radius $r$ of each $V_n$ contain at most $k^{r+1}$ elements. Therefore, 
\[2A_n\geq \frac{1}{|V_n|^2} \sum_{v,u\in V_n}\|\varphi_n(v)-\varphi_n(u)\|\geq \left(1-\frac{k^{r+1}}{|V_n|}\right)\rho(r).\]
Letting $n\to \infty$ and then $r\to \infty$, we obtain that  $\lim_{n\to\infty}A_n=\infty$.

The proof now  consists of truncating the functions $\varphi_n$ so that they lie in the balls $A_nB_{X_n}$, translating them so that their mean is zero, and then rescaling them appropriately. For each $n\in\N$, let $R_n\colon X_n\to A_nB_{X_n}$ denote the radial truncation
\begin{equation*}
    R_n(x)=\left\{\begin{array}{cc}
      x   ,&  \text{ if }\ \|x\|\leq A_n,\\
       A_n\frac{x}{\|x\|},  & \text{ if }\ \|x\|>A_n. 
    \end{array}\right.
\end{equation*}
It is straightforward to check that $\Lip(R_n)\leq 2$. For $n\in\N$ and $v\in V_n$, define 
\[\theta_n(v)=R_n(\varphi_n(v))\ \text{ and }\ m_n=\frac{1}{|V_n|}\sum_{u\in V_n}\theta_n(u).\]
Let   $C=2kL/h$ and define each $\psi_n\colon V_n\to B_{X_n}$ by letting \[\psi_n(v)=\frac{\theta_n(v)-m_n}{A_n+C}\   \text{ for all } \ v\in V_n.\]
Let us now show that these maps have the required properties. The fact that $\sum_{v\in V_n}\psi_n(v)=0$ is immediate and, as $\lim_{n\to \infty}A_n=\infty$, we also have that 
\[\Lip(\psi_n)\leq \frac{2L}{A_n+C}\underset{n\to \infty}{\longrightarrow} 0.\]

We are left to show that each $\psi_n$ takes values in $B_{X_n}$ and that the limit inferior of  $\tfrac{1}{|V_n|}\sum_{v\in V_n}\|\psi_n(v)\|$  is at least $1$.  As each $v\in V_n\mapsto \|\varphi_n(v)\|\in \mathbb{R}$ has Lipschitz constant at most $L$, \eqref{PIII} and the choice of $C$ give that 
\begin{equation}\label{Eq.piii}
    \frac{1}{|V_n|}\sum_{v\in V_n}\left|\|\varphi_n(v)\|-A_n\right|\leq C.
\end{equation}
Since $\sum_{v\in V_n}\varphi_n(v)=0$, \eqref{Eq.piii} implies that
\begin{equation*}
    \|m_n\|\leq\frac{1}{|V_n|}\sum_{v\in V_n}\|\varphi_n(v)-\theta_n(v)\|\leq \frac{1}{|V_n|}\sum_{v\in V_n}|\|\varphi_n(v)\|-A_n|\leq C.
\end{equation*}
By the formula of $\psi_n$, this  implies that these maps take values in $B_{X_n}$ as needed.
Also, since 
\[\|\theta_n(v)\|=\min\{\|\varphi_n(v)\|,A_n\}\geq \|\varphi_n(v)\|-|\|\varphi_n(v)\|-A_n|,\]
 \eqref{Eq.piii}  implies that
\begin{equation*}
    \frac{1}{|V_n|}\sum_{v\in V_n}\|\theta_n(v)\|\geq     \frac{1}{|V_n|}\sum_{v\in V_n}\|\varphi_n(v)\|-    \frac{1}{|V_n|}\sum_{v\in V_n}|\|\varphi_n(v)\|-A_n|\geq A_n-C.
\end{equation*}
We then conclude that 
\begin{align*}
     \frac{1}{|V_n|}\sum_{v\in V_n}\|\psi_n(v)\|\geq \frac{\frac{1}{|V_n|}\sum_{v\in V_n}\|\theta_n(v)\|-C}{A_n+C}\geq \frac{A_n-2C}{A_n+C}\underset{n\to \infty}{\longrightarrow}1,
\end{align*}
as desired.
\end{proof}
 
Another trivial observation needed for the proof of Theorem \ref{Thm.main} is the following sufficient condition for vanishing degree.

\begin{proposition}\label{Prop.null.homotopy}
    Let $X$ and $Y$ be finite dimensional normed   spaces and $f\colon B_X\to B_Y$ be a continuous map such that $\inf_{x\in B_X}\|f(x)\|>0$. Then the map $F\colon S_X\to S_Y$ given by $F(x)=f(x)/\|f(x)\|$ has degree zero.
\end{proposition}
 
\begin{proof}
Let $g_s\colon B_X\to B_X$, $s\in [0,1]$, be a homotopy from $g_0=\mathrm{Id}_{B_X}$ to the constant map $g_1\equiv 0$. Then the map $G_s(x)=f(g_s(x))/\|f(g_s(x))\|$ defines a homotopy in $S_Y$ from $G_0=F$ to the constant map $G_1\equiv f(0)/\|f(0)\|$. Since null-homotopic maps have degree zero, the conclusion follows.
\end{proof}

\begin{proof}[Proof of Theorem \ref{Thm.main}]
We start fixing the objects given in the statement of the theorem together with Lemma \ref{Lemma}. Discarding finitely many terms, we may assume that $|V_n|\geq5$ for all $n\in\N$. Let  $k\in\N$ and $h>0$ be such that each $G_n=(V_n,E_n)$ is a $(k,h)$-expander. For each $n\in\N$ let $d_n$ be the shortest path metric on $V_n$. Let $(\psi_n\colon V_n\to B_{X_n})_n$ be the maps given by Lemma \ref{Lemma} due to the existence of  equi-coarse embeddings $(V_n\to X_n)_n$ and recall that  $(F_n\colon S_{X_n}\to S_{Y_n})_n$ denotes a  sequence of equi-uniformly continuous maps.

We  now  show how to use Proposition \ref{Prop.null.homotopy} in order to obtain the vanishing of the degree of $F_n$ for large $n$. The proof has two parts: (1) we first show each $F_n$ to be homotopic to some other map, and then (2) we show that this other map is the normalization of a function on $B_{X_n}$ which avoids zero. We start extending each $F_n$ radially to $G_n:2B_{X_n}\to2B_{Y_n}$ by letting \[
  G_n(x)=
  \begin{cases}
    \norm{x}F_n\left(\frac{x}{\norm{x}}\right),&x\neq0,\\
    0,&x=0.
  \end{cases}
\]
Clearly,   each $G_n$ is  norm preserving and its restriction to $S_{X_n}$ is $F_n$. It is also straightforward to check that, as $(F_n)_n$ are equi-uniformly continuous,  so are $(G_n)_n$ (e.g., see \cite[Proposition 2.9]{OdellSchlumprecht1994Acta}). Fix then an increasing $\omega\colon [0,\infty)\to [0,\infty)$  with $\lim_{t\to 0}\omega(t)=0$ such that
\[\|G_n(x)-G_n(y)\|\leq \omega(\|x-y\|)\ \text{ for all }\ n\in\N\ \text{ and  }x,y\in 2B_{X_n}.\]

  For each  $n\in\N$,   $ 
  x\in B_{X_n}$, and  $s\in[0,1]$,  let 
\begin{equation*}
 H^n_s(x)=\frac{1}{|V_n|}\sum_{v\in V_n} G_n(x+s\psi_n(v)).
  \label{eq:average}
\end{equation*}
Then, for each  $x\in B_{X_n}$ and $s\in[0,1]$, applying \eqref{PIII} to the map   
\[
v\in V_n\mapsto G_n(x+s\psi_n(v))\in Y_n\subseteq  L_1 ,
\]
we obtain that, for all $n\in\N$,
\begin{equation*}
  \begin{aligned}
   \frac{1}{|V_n|} \sum_{v\in V_n} \norm{G_n(x+s\psi_n(v))-H^n_s(x)}
    &\leq \frac{2k}{h}\omega\left(\Lip(\psi_n)\right).
  \end{aligned}
  \label{eq:deviation}
\end{equation*}
As each $G_n$ is norm preserving, this inequality gives that
\begin{align}
   \norm{H^n_s(x)}&\geq \frac{1}{|V_n|}\sum_{v\in V_n}\norm{x+s\psi_n(v)}-\frac{2k}{h}\omega\left(\Lip(\psi_n)\right)
  \label{eq:key}
\end{align}
for all  $n\in\N$, all 
      $x\in B_{X_n}$, and all  $ s\in[0,1]$.
As $ \sum_{v\in V_n}\psi_n(v)=0$ (see \eqref{Eq.prop.psi}), 
\begin{align*}
  \frac{1}{|V_n|}\sum_{v\in V_n}\norm{x+s\psi_n(v)}\geq 
   \left\|x+s\frac{1}{|V_n|}\sum_{v\in V_n}\psi_n(v)\right\|  \geq\left\|x\right\|, 
\end{align*}
and it follows that 
  \begin{equation} \label{Eq.lowerboundH1-eps}          \norm{H^n_s(x)} \geq 1-\frac{2k}{h}\omega\left(\Lip(\psi_n)\right)\ \text{ for all }\ x\in S_{X_n}\ \text{ and }\ s\in [0,1].\end{equation}

By our choice of $(\psi_n)_n$,  $\lim_n\Lip(\psi_n)=0$ (see \eqref{Eq.prop.psi}). Hence, since  $\lim_{t\to 0}\omega(t)=0$, we can  pick $n_0\in\N$ such that   \linebreak$4k\omega (\Lip(\psi_n))<h$ for all $n\geq  n_0$. By \eqref{Eq.lowerboundH1-eps},   $\norm{H^n_s(x)} \geq 1/2$ for all $n\geq n_0$, $x\in S_{X_n}$, and $s\in [0,1]$. We can then define,  for any given $n\geq n_0$,
\begin{equation}
  f^n_s(x)=\frac{H^n_s(x)}{\norm{H^n_s(x)}}, \ \text{ for all }\ x\in S_{X_n} \ \text{ and }\ s\in  [0,1].
  \label{eq:homotopy}
\end{equation}
So, $(f^n_s)_{s\in [0,1]}$ is a homotopy starting at $f^n_0=F_n$. This completes the first step in our approach towards applying Proposition \ref{Prop.null.homotopy}. 

We are  left to show that  $f^n_1$ has degree zero for $n$ sufficiently large; for this we show that $H^n_1$ is bounded away from zero on $B_{X_n}$. Notice that, using that $\sum_{v\in V_n} \psi_n(v)=0$ once again,
\begin{align*}
  \frac{1}{|V_n|}\sum_{v\in V_n}\norm{\psi_n(v)} & \leq\frac{1}{|V_n|}\sum_{v\in V_n}\norm{x+\psi_n(v)}+\norm{x}\\
  &\leq\frac{1}{|V_n|}\sum_{v\in V_n}\norm{x+\psi_n(v)}+\norm{x+\frac{1}{|V_n|}\sum_{v\in V_n}\psi_n(v)}\\
 &\leq  \frac{2}{|V_n|}\sum_{v\in V_n}\norm{x+\psi_n(v)}.
\end{align*}
Therefore, applying  \eqref{eq:key} with $s=1$, we get that 
\begin{equation}
 \norm{H^n_1(x)}\geq  \frac{1}{2|V_n|}\sum_{v\in V_n}\norm{\psi_n(v)}-\frac{2k}{h}\omega\left(\Lip(\psi_n)\right)\ \text{ for all }\ x\in B_{X_n}.
  \label{eq:ball}
\end{equation}

By our choice of $(\psi_n)_n$ (see \eqref{Eq.prop.psi}), we can pick  $n_1\geq n_0$ such that
\[\frac{1}{2|V_n|}\sum_{v\in V_n}\norm{\psi_n(v)}\geq \frac{1}{4}\ \text{ and }\ \frac{2k}{h}\omega\left(\Lip(\psi_n)\right)\leq\frac{1}{8}\ \text{ for all } \ n\geq n_1.\]
Then  \eqref{eq:ball} gives that 
\begin{equation*}
 \norm{H^n_1(x)}\geq\frac{1}{8}\ \text{ for all }\ x\in B_{X_n}.
\end{equation*}
By Proposition \ref{Prop.null.homotopy}, this shows that $f^n_1$ has degree zero for all $n\geq n_1$.\end{proof}

% \subsection{AI usage disclaimer}
% Theorem \ref{Thm.main} was obtained by  GPT-6 Astra with little input from the authors.   The organization and exposition of the paper are the authors' responsibility only. Moreover, all work done in collaboration with GPT-6 Astra  was independently checked and reworked by the authors,   who also take full responsibility for the correctness  and integrity of the work. 

\subsection{AI usage disclaimer}
A proof of Theorem \ref{Thm.main} was obtained by the authors using GPT-6 Astra with little of their own input. The first proposed proof used a family of Markov chains on the $\ell_\infty$-grids $\{-m,\dots, m\}^n$ that have uniform spectral gaps, with each chain being an $n$-fold $\ell_\infty$-product of a stationary, reversible chain on $\{-m,\dots, m\}$ with spectral gap independent of $m$. The authors recognized that a spectral gap was essentially all that was needed for the proof to go through and prompted GPT-6 for a new proof that would work for any Banach space coarsely containing bounded degree expanders. The authors independently verified the arguments and wrote the proofs and the rest of the article in their preferred style. The authors take full responsibility for the correctness and integrity of the work.

\subsection{Lean formalization}
Lean formulations of the main theorem and its corollaries,
together with their supporting proofs, were produced using
GPT-6 Astra and checked in Lean 4 (version 4.33.1), using
Mathlib and additional source files included in the repository.
The source files and verification instructions are available at
\begin{center}
\url{https://github.com/demendoncabraga/propertyH_lean}
\end{center}
\noindent
The only external mathematical result supplied as an explicit
hypothesis is Osajda's \cite[Theorem 4]{Osajda2020Acta}, used
in Corollary \ref{Cor.Group}.

\subsection{Disclaimer on originality of the work} \label{disclaimer}
A proof of the failure of rational Property (H) for $c_0$ appeared on arXiv on September 15th, 2026 in \cite{ChengChengWang}. We obtained the proofs of the results in this paper by September 13th, 2026, independent of \cite{ChengChengWang}. While results being obtained independently by different groups has always been part of mathematical research, we believe that this may become more common with the advent of new technologies. After some consideration, we decided to still present this paper as our final draft is short and the main results, Theorem~\ref{Thm.main} and Corollaries~\ref{Cor.Prop.H} and \ref{Cor.Group}, cover a potentially larger class of Banach spaces than just those which contain the finite dimensional $(\ell_\infty^n)_n$ with uniform distortion.
 
% \begin{disclaimer*}
% Other proofs of $c_0$ failing Property (H)
% and of Corollary~\ref{Corollary.Johnson} appeared on arXiv on September 15th, 2026 in \cite{ChengChengWang}. We obtained the proofs of the results in this paper by September 13th, 2026, independent of \cite{ChengChengWang}. While results being obtained independently by different groups has always been part of mathematical research, we believe that this may become more common with the advent of new technologies. After some consideration, we decided to still present this paper as our final draft is short and the main results, Theorem~\ref{Thm.main} and Corollary~\ref{Cor.Prop.H}, cover a potentially larger class of Banach spaces than just those of trivial cotype.
% \end{disclaimer*}

\subsection{Acknowledgements} This paper was written under the auspices of the
American Institute of Mathematics (AIM) SQuaREs program as part of the
``Nonlinear geometry of Banach spaces''  SQuaRE project. The authors are also indebted to Cynthia Bortolotto and João Pedro Ramos for sharing their files for Lean formalization; without this, the Lean formalization would not have been possible. B. M. Braga would also like to thank  Christian Rosendal,  Mitchel Taylor, and Cláudio Verdun for several discussions about Lean formalization.

\bibliographystyle{amsalpha}
 \bibliography{bibliography}

@article{KasparovYu2012GeoTop,
author = {G. Kasparov and G. Yu},
title = {{The Novikov conjecture and geometry of Banach spaces}},
volume = {16},
journal = {Geometry \& Topology},
number = {3},
publisher = {MSP},
pages = {1859 -- 1880},
year = {2012},
doi = {10.2140/gt.2012.16.1859},
URL = {https://doi.org/10.2140/gt.2012.16.1859}
}

@article {SkandalisTuYu2002Top,
    AUTHOR = {Skandalis, G. and Tu, J. L. and Yu, G.},
     TITLE = {The coarse {B}aum-{C}onnes conjecture and groupoids},
   JOURNAL = {Topology},
  FJOURNAL = {Topology. An International Journal of Mathematics},
    VOLUME = {41},
      YEAR = {2002},
    NUMBER = {4},
     PAGES = {807--834},
      ISSN = {0040-9383},
   MRCLASS = {58J22 (19K56 46L80 46L85)},
  MRNUMBER = {1905840},
MRREVIEWER = {Paul\ D.\ Mitchener},
       DOI = {10.1016/S0040-9383(01)00004-0},
       URL = {https://doi.org/10.1016/S0040-9383(01)00004-0},
}

@article {Gromov2003GAFA,
    AUTHOR = {Gromov, M.},
     TITLE = {Random walk in random groups},
   JOURNAL = {Geom. Funct. Anal.},
  FJOURNAL = {Geometric and Functional Analysis},
    VOLUME = {13},
      YEAR = {2003},
    NUMBER = {1},
     PAGES = {73--146},
      ISSN = {1016-443X,1420-8970},
   MRCLASS = {20F65 (20F67 20P05 60G50)},
  MRNUMBER = {1978492},
MRREVIEWER = {Thomas\ Delzant},
       DOI = {10.1007/s000390300002},
       URL = {https://doi.org/10.1007/s000390300002},
}

@ARTICLE{ChengChengWang,
       author = {{Cheng}, L. and {Cheng}, Q.  and {Wang}, Y.},
        title = "{On Property (H) of $c_0$ and the Sphere Problems of Gromov and Johnson}",
      journal = {arXiv e-prints},
         year = 2026,
        month = sep,
          eid = {arXiv:2609.15817},
        pages = {arXiv:2609.15817},
archivePrefix = {arXiv},
       eprint = {2609.15817},
 primaryClass = {math.FA},
       adsurl = {https://ui.adsabs.harvard.edu/abs/2026arXiv260915817C}
}

@book {Ostrovskii2013Book,
    AUTHOR = {Ostrovskii, M.},
     TITLE = {Metric embeddings},
    SERIES = {De Gruyter Studies in Mathematics},
    VOLUME = {49},
      NOTE = {Bilipschitz and coarse embeddings into Banach spaces},
 PUBLISHER = {De Gruyter, Berlin},
      YEAR = {2013},
     PAGES = {xii+372},
      ISBN = {978-3-11-026340-4; 978-3-11-026401-2},
   MRCLASS = {46-01 (46-02 46B20 46B85)},
  MRNUMBER = {3114782},
MRREVIEWER = {Florent\ Baudier},
       DOI = {10.1515/9783110264012},
       URL = {https://doi.org/10.1515/9783110264012},
}

@book {LubotzkyBook2010,
    AUTHOR = {Lubotzky, A.},
     TITLE = {Discrete groups, expanding graphs and invariant measures},
    SERIES = {Modern Birkh\"{a}user Classics},
      NOTE = {With an appendix by Jonathan D. Rogawski,
              Reprint of the 1994 edition},
 PUBLISHER = {Birkh\"{a}user Verlag, Basel},
      YEAR = {2010},
     PAGES = {iii+192}
}

@article {OdellSchlumprecht1994Acta,
    AUTHOR = {Odell, E. and Schlumprecht, Th.},
     TITLE = {The distortion problem},
   JOURNAL = {Acta Math.},
  FJOURNAL = {Acta Mathematica},
    VOLUME = {173},
      YEAR = {1994},
    NUMBER = {2},
     PAGES = {259--281},
      ISSN = {0001-5962,1871-2509},
   MRCLASS = {46B20},
  MRNUMBER = {1301394},
MRREVIEWER = {G.\ Schechtman},
       DOI = {10.1007/BF02398436},
       URL = {https://doi.org/10.1007/BF02398436},
}

@article {ChengWang2018JMAA,
    AUTHOR = {Cheng, Q. and Wang, Q.},
     TITLE = {On {B}anach spaces with {K}asparov and {Y}u's {P}roperty
              ({H})},
   JOURNAL = {J. Math. Anal. Appl.},
  FJOURNAL = {Journal of Mathematical Analysis and Applications},
    VOLUME = {457},
      YEAR = {2018},
    NUMBER = {1},
     PAGES = {200--213},
      ISSN = {0022-247X,1096-0813},
   MRCLASS = {46E20},
  MRNUMBER = {3702702},
MRREVIEWER = {Isabel\ Marrero},
       DOI = {10.1016/j.jmaa.2017.08.005},
       URL = {https://doi.org/10.1016/j.jmaa.2017.08.005},
}

@article{Yu2000,
	author = {Yu, G.},
	doi = {10.1007/s002229900032},
	fjournal = {Inventiones Mathematicae},
	journal = {Invent. Math.},
	number = {1},
	pages = {201--240},
	title = {The coarse {B}aum-{C}onnes conjecture for spaces which admit a uniform embedding into {H}ilbert space},
	volume = {139},
	year = {2000}}

@article {Osajda2020Acta,
    AUTHOR = {Osajda, D.},
     TITLE = {Small cancellation labellings of some infinite graphs and
              applications},
   JOURNAL = {Acta Math.},
  FJOURNAL = {Acta Mathematica},
    VOLUME = {225},
      YEAR = {2020},
    NUMBER = {1},
     PAGES = {159--191},
      ISSN = {0001-5962,1871-2509},
   MRCLASS = {05C63 (05C78 20F06 20F65 57M15)},
  MRNUMBER = {4176066},
MRREVIEWER = {Jingyin\ Huang},
       DOI = {10.4310/acta.2020.v225.n1.a3},
       URL = {https://doi.org/10.4310/acta.2020.v225.n1.a3},
}

\end{document}